\documentclass{article}
\usepackage{amsmath, amssymb, amsthm, amsrefs}
\usepackage{verbatim}
\usepackage{hyperref}
\usepackage{ellipsis}
\usepackage{enumerate}
\numberwithin{equation}{section}

\newcommand{\nc}{\newcommand}
\nc{\Qq}{\mathbb{Q}}
\nc{\Cc}{\mathbb{C}}
\nc{\Rr}{\mathbb{R}}
\nc{\Zz}{\mathbb{Z}}
\nc{\Nn}{\mathbb{N}}
\nc{\Pp}{\mathbb{P}}

\nc{\cA}{\mathcal{A}}
\nc{\cP}{\mathcal{P}}
\nc{\cC}{\mathcal{C}}
\nc{\cB}{\mathcal{B}}
\nc{\cE}{\mathcal{E}}
\nc{\cF}{\mathcal{F}}
\nc{\cS}{\mathcal{S}}
\nc{\cH}{\mathcal{H}}
\nc{\cZ}{\mathcal{Z}}
\nc{\cM}{\mathcal{M}}
\nc{\cN}{\mathcal{N}}
\nc{\cI}{\mathcal{I}}
\nc{\cT}{\mathcal{T}}

\nc{\mfk}{\mathfrak} \nc{\mf}{\mfk{m}} \nc{\vp}{ \vec{p}}
\nc{\vD}{\vec{\Delta}}
\nc{\vm}{\vec{m}}
\nc{\vl}{\vec{\ell}}
\nc{\vk}{\vec{k}}
\nc{\va}{\vec{a}}
\nc{\vc}{\vec{c}}

\nc{\norm}[1]{\left\| #1 \right\|} \nc{\gen}[1]{\langle #1 \rangle} 

\nc{\llrrbrac}[1]{%
\left[\mkern-1mu\left[#1\right]\mkern-1mu\right]}

\nc{\llrrpar}[1]{%
\left(\mkern-1mu\left(#1\right)\mkern-1mu\right)}

\nc{\wL}[3]{\det
  \begin{pmatrix} #1 & \cdots & #2 \\
  \partial #1 & \cdots & \partial #2 \\ \vdots & \phantom{\vdots} &
  \vdots \\ \partial^{#3} #1 & \cdots & \partial^{#3} #2
\end{pmatrix}}

\nc{\wron}[3]{\det
  \begin{pmatrix} #1 & \cdots & #2 \\
  #1^\prime & \cdots & #2^\prime \\ 
  \vdots & \phantom{\vdots} & \vdots \\ 
  #1^{(#3)} & \cdots & #2^{(#3)}
\end{pmatrix}}

\nc{\dO}{\partial_{\Omega}}
\nc{\id}{\mathrm{id}}

\nc{\DMO}{\DeclareMathOperator}
\renewcommand{\vec}{\mathbf}

\DMO{\ord}{ord}

\newtheorem*{thm*}{Theorem}
\newtheorem{thm}{Theorem}[section]
\newtheorem{lem}[thm]{Lemma}
\newtheorem{prop}[thm]{Proposition}

\theoremstyle{remark}

\renewcommand{\MR}[1]{}

\begin{document}
\title{Wronskians and Linear Dependence of Formal Power Series}

%\markright{Wronskians and linear dependence}

\author{ Keith Ball\thanks{Partially supported by NSF Grant-DMS-1247679,
Project PUMP: Preparing Undergraduates through Mentoring towards PhDs.}
\and Cynthia Parks\footnotemark[1] \and Wai Yan Pong\footnotemark[1] \\
California State University Dominguez-Hills}

\maketitle 

\begin{abstract}
 We give a new proof of the fact that the vanishing of generalized Wronskians
 implies linear dependence of formal power series in several variables. Our
 results are also valid for quotients of germs of analytic functions.
\end{abstract}

\section{Introduction} \label{sec:intro}

Suppose $f_1, \ldots, f_n$ are $(n-1)$-times differentiable functions. If they
are linearly dependent over the constants then their Wronskian vanishes, i.e.
\begin{align*}
  \wron{f_1}{f_n}{n-1} = 0. \qquad \tag{$*$}
\end{align*} 
The converse of this classical fact, however, is not true, even for
$C^{\infty}$-functions. An example illustrating this was given over a century
ago by B\^ocher~\cite{bocher}: the functions
\begin{align*}
  f_1 = 
  \begin{cases}
    e^{-1/x^2} & x \neq 0; \\
    0 & x = 0
  \end{cases}, \quad 
  \text{and} \quad
  f_2 = 
  \begin{cases}
    e^{-1/x^2} & x > 0; \\
    0 & x = 0\\
    2e^{-1/x^2} & x < 0
  \end{cases}
\end{align*}
are linear independent over $\Rr$, yet their Wronskian vanishes identically.
Over the years, a number of variations of the converse of $(*)$ had been
considered~\cites{wolsson, walker, bd, wdww}. In the several variables case,
with Wronskian replaced by generalized Wronskians, the converse of $(*)$
was proved for analytic functions. Bostan and Dumas gave a short proof
of this result in~\cite{bd}. They proved, more generally, that if $K$
is a field of characteristic zero and formal power series $f_1,\ldots,
f_n \in K\llrrbrac{X_1,\ldots, X_m}$ are linearly independent over $K$,
then at least one of their generalized Wronskians is not identically zero
(\cite{bd}*{Theorem~3}). In this article we propose an equally short, yet
completely different proof of this result. We do so by generalizing it to,
$F_K$, the field of fractions of the formal power series rings over $K$
in countably many indeterminates. More precisely, we prove:
\begin{thm}
  Let $K$ be a characteristic zero field. A finite family $f_1,\ldots, f_n$
  of elements of $F_K$ is linearly dependent over $K$ if and only if all
  generalized Wronskians of the family vanish.  \label{th:main}
\end{thm}
Our proof of Theorem~\ref{th:main} takes advantage of the fact that the
converse of $(*)$ holds for differential fields~\cite{afg}*{Theorem~6.3.4}:
\begin{thm}
  Let $D$ be a derivation of a field $F$. A finite family of $f_1,\ldots
  f_n$ of elements of $F$ is linearly dependent over the kernel of $D$ if
  and only if the Wronskian of $f_1,\ldots, f_n$ with respect to $D$ is 0.
  \label{th:wdf}
\end{thm}
We will show in Proposition~\ref{p:all=>log} that the vanishing of all
generalized Wronskians of a family in $F_K$ implies the vanishing of its
Wronskian with respect to the log-derivation which can be defined if $K$
contains a copy of $\Qq(\log(k) \colon k \in \Nn)$. So for those fields $K$
Theorem~\ref{th:main} follows from Theorem~\ref{th:wdf} since the kernel of the
log-derivation is $K$ (Proposition~\ref{p:constants}). The general case can be
argued over an extension of $K$ that contains $\Qq(\log(k) \colon k \in \Nn)$
then use the fact that linear independence is preserved under extension of
scalars. We find borrowing a transcendental function, namely the logarithm,
to prove a purely algebraic result interesting. In the next section the
reader will see that considering the problem in countably many variables
allows us to examine it through the lens of arithmetic functions/Dirichlet
series. It is only from that point of view that the relevance of logarithm
becomes transparent.

\section{Ring of Arithmetic Functions} \label{sec:ps-af}

Let $K$ be a field. By a {\bf $K$-valued arithmetic function} we mean a
function from $\Nn$ to $K$. With $\alpha \in K$ identified with the function
$1 \mapsto \alpha$ and $n \mapsto 0$ for all $n >1$ the $K$-valued arithmetic
functions form a $K$-algebra, denoted by $A_K$, under the operations:
\begin{enumerate}
  \item $(f+g)(n) := f(n) + g(n)$; and
  \item $\displaystyle (f*g)(n) := \sum_{mk=n}f(m)g(k)$.
\end{enumerate}
The following are a few specific $\Qq$-valued arithmetic functions that often
appear in this article: let $e_n$ be the function whose value is $1$ at $n$
and $0$ elsewhere. For each prime $p$, let $v_p$ be the function that assigns
to each $n$ the largest integer $k \ge 0$ so that $p^k$ divides $n$. Let
$\Omega$ be the function $\sum_{p \in \Pp} v_p$ where $p$ runs through the
set of primes $\Pp$, thus $\Omega$ counts the total number of prime factors
(with multiplicity) of its argument.

At first sight the ring $A_K$ does not look like a ring of power series, but it
is actually isomorphic, as $K$-algebra, to $K\llrrbrac{X_p \colon p \in \Pp}$
via
\begin{equation}
  f \longmapsto \sum_{n \in \Nn} f(n)\prod_{p \in \Pp} X_p^{v_p(n)}.
  \label{eq:iso}
\end{equation}
The algebras $K\llrrbrac{X_p \colon p \in \Pp}$ and $K\llrrbrac{X_i \colon i
\in \Nn}$ are certainly isomorphic too, but our results can be stated much
more elegantly if $A_K$ is identified with the former algebra. Note also
that $A_K$ is isomorphic to the algebra of formal Dirichlet series with
coefficients from $K$ via
\begin{equation}
  f \longmapsto \sum_{n \ge 1} \frac{f(n)}{n^s}.
  \label{eq:iso_dirichlet}
\end{equation}
We will drop the base field $K$ from the notation if no confusion arises.

For a nonzero $f \in A$, let $\ord(f)$ be the least $n \in \Nn$ such that
$f(n) \neq 0$. We set $\ord(0) = \infty$. Let $\norm{f}$ to be $1/\ord(f)$
(with $1/\infty =0$). Clearly $\norm{f} \ge 0$ for any $f \in A$ and the
equality holds precisely when $f =0$. It is also clear that $\norm{f+g} \le
\max\{\norm{f}, \norm{g}\}$. Moreover, one checks that for all $f,g \in A$,
$(f*g)(n) = 0$ for all $n < \ord(f)\ord(g)$, and that
\begin{align*}
  (f*g)(\ord(f)\ord(g)) = f(\ord(f))g(\ord(g)).
\end{align*}
Therefore, $\ord(f*g) =\ord(f)\ord(g)$ or equivalently $\norm{f*g} =
\norm{f}\norm{g}$. Thus $\norm{\cdot}$ is an ultrametric absolute value
on $A$ (c.f.~\cite{vf}*{Section~1.1}). Consequently, $A$ is an integral
domain. Let $F$ denote its field of fractions. Note that $\norm{\cdot}$
extends uniquely to an ultrametric absolute value of $F$ by $\norm{f/g}:=
\norm{f}/\norm{g}$. We further extend $\norm{\cdot}$ to a norm on $F^m$ ($m
\in \Nn$) by setting $\norm{\vec{x}}:= \max\{\norm{x_i} \colon 1 \le i \le
m\}$. A sequence $(\vec{x}_n)$ in $F^m$ {\bf converges} to an element $\vec{x}
\in F^m$, if the sequence $(\norm{\vec{x}_n - \vec{x}})_{n \in \Nn}$ converges
to $0$. Addition and multiplication of $F$ are continuous (i.e. preserving
convergent sequences) and so are the coordinate projections. As a result,
polynomial functions are continuous. In particular for each $n$, the
determinant function from $F^{n\times n}$ to $F$ is continuous.

A {\bf derivation} of a ring $R$ is a map $D$ from $R$ to itself satisfying
$D(a+b) = D(a) + D(b)$ and $D(ab) = aD(b) + bD(a)$ for all $a,b \in R$. We
will not distinguish by notation a derivation of an integral domain from its
unique extension (by the ``quotient rule'') to the field of fractions. There
are several ``natural'' derivations of $A$ (and hence $F$) coming from the
isomorphisms~\eqref{eq:iso} and~\eqref{eq:iso_dirichlet}: the derivation
$\partial/\partial X_p$ $(p \in \Pp)$ of $K\llrrbrac{X_p \colon p \in \Pp}$
corresponds to the derivation $\partial_p$ of $A$ given by
\begin{align*}
  (\partial_p f) (n) = v_p(np)f(np), \quad (f \in A, n \in \Nn).
\end{align*}
We call each $\partial_p$ ($p \in \Pp$) a {\bf basic derivation} of $A$. If
$K$ contains a copy of the field $\Qq(\log(k) \colon k \in \Nn)$, then the
derivation $-d/ds$ of the ring of formal Dirichlet series corresponds to
the derivation $\partial$ of $A$ given by
\begin{equation*}
  \partial f(n) = \log(n)f(n),\quad (f \in A, n \in \Nn) \label{eq:log-der}.
\end{equation*}
We call $\partial$ the {\bf log-derivation} of $A$. It is easy to check that
$\norm{f} = \norm{\partial f}$ if and only if $\norm{f} < 1$, or equivalently
$f$ is not a unit of $A$. One checks also that the kernel of $\partial$ in
$A$, i.e. the set $\ker_A \partial:=\{f \in A \colon \partial f = 0\}$, is
$K$. Next we show that extending $\partial$ to $F$ does not enlarge its kernel.
\begin{prop}
  $\ker_{F_K} \partial = K$.  \label{p:constants}
\end{prop} 
\begin{proof}
  The inclusion $K \subseteq \ker_F \partial$ is clear. To establish the
  reverse inclusion, pick any $f/g \in \ker_{F} \partial \setminus\{0\}$ then
  \begin{align}
    \label{eq:quo} \partial f * g = f * \partial g 
  \end{align} 
  and so $\norm{\partial f}\norm{g} = \norm{f}\norm{\partial g}$. If $g$
  is invertible in $A$, then $f/g$ is already in $\ker_A \partial =K$. So
  suppose $g$ is not a unit, it then follows that $\norm{g} = \norm{\partial
  g}$ and hence $\norm{f} = \norm{\partial f}$ (because $\norm{g} \neq
  0$). Evaluating both sides of~\eqref{eq:quo} at $\ord(f)\ord(g)$ yields
  $\log(\ord(f)) = \log(\ord(g))$ and hence $\ord(f) = \ord(g)$. Consider
  the function $h:=f-(f(n)/g(n))g$ where $n=\ord(f)=\ord(g)$. Then $\ord(h)
  > \ord(g)$ but $h/g = f/g - f(n)/g(n) \in \ker_F \partial$. So $h$ must
  be 0, i.e. $f/g = f(n)/g(n) \in K$, otherwise the same argument with $f$
  replaced by $h$ will show that $\ord(h)=\ord(g)$, a contradiction. This
  completes the proof of the other inclusion.
\end{proof}
It is probably worth pointing that the validity of the proposition hinges
on the fact that $\log$ is 1-to-1. For instance, one checks readily that
the kernel of the derivation on $A$ given by $\partial_{\Omega}f(n) =
\Omega(n)f(n)$ is $K$ but for distinct primes $p,q$, $e_p/e_q \in \ker_F
\partial_{\Omega} \setminus K$.

Every continuous derivation of $A$ can be expressed as a series of basic
derivations~\cite{conv}*{Theorem~4.3}. The series for the log-derivation is
\begin{equation}
 \sum_{p \in \Pp} \log(p)e_p*\partial_p.
\label{eq:der-relation}
\end{equation}
That means for each $f \in A$, the sequence $(s_N(f))_{N \in \Nn}$
converges to $\partial f$ where $s_N$ is the partial sums $\sum_{p \le N}
\log(p)e_p*\partial_p$. Next we show that the same relation holds for their
extensions to $F$.
\begin{lem}
  As derivations of $F$, $\partial = \sum_{p} \log(p)e_p*
  \partial_p$. \label{l:log-basic}
\end{lem}
\begin{proof}
  We need to show that for any $f$ and $g\neq 0$ in $A$, $\norm{s_N(f/g) -
  \partial(f/g)} \to 0$ as $N \to \infty$. Since $s_N$ is a derivation,
  \begin{align*}
    &\norm{s_N\left(\frac{f}{g}\right) - \partial\left( \frac{f}{g} \right)}\\
    =&\norm{\frac{s_N(f) - \partial f}{g} - \left( \frac{f}{g}
    \right)\frac{s_N(g) - \partial g}{g}} \\
    \le& \max\left\{ \frac{\norm{s_N(f) - \partial f}}{\norm{g}},
    \frac{\norm{f}\norm{s_N g - \partial g}}{\norm{g}^2} \right\}.
  \end{align*}
  The maximum tends to 0 as $N$ tends to $\infty$ because $s_N$ converges
  to $\partial$ on $A$.
\end{proof}

Basic derivations commute with each other. In general, two derivations need
not commute but their commutator is always a derivation. The next lemma can
be viewed as a generalization of this fact to differential operators of the
form $\partial_m:=\prod_{p \in \Pp} \partial_p^{v_p(m)}$ ($m \in \Nn$). Note
that $\partial_1$ is simply the identity map of $F$.
\begin{lem}
   $[\partial_m, \partial] = \log(m)\partial_m$ on $F$.  \label{l:commutator}
\end{lem}
\begin{proof}
  We prove this by induction on $\Omega(m)$. The case $\Omega(m) = 0$,
  i.e. $m=1$ is clear. Assume for some $k \ge 0$, the lemma is true for all
  $n$ with $\Omega(n) = k$. Pick any $m \in \Nn$ with $\Omega(m) = k+1$ then
  $m =np$ for some prime $p$ and $n$ with $\Omega(n) = k$. So $[\partial_n,
  \partial] = \log(n)\partial_n$ by the induction hypothesis. One checks
  directly that the derivations $[\partial_p,\partial]$ and $\log(p)\partial_p$
  agree on $A$ and hence on $F$. Thus, the following computation finishes
  the proof.
  \begin{align*}
    [\partial_m,\partial] &= \partial_m\partial - \partial\partial_m =
    \partial_p(\partial_n\partial) - (\partial\partial_p)\partial_n \\
    &= \partial_p(\partial\partial_n + \log(n)\partial_n) - (\partial_p\partial
    - \log(p)\partial_p)\partial_n \\
    &= \partial_p\partial\partial_n + \log(n) \partial_p\partial_n -
    \partial_p\partial\partial_n + \log(p)\partial_p\partial_n =
    \log(m)\partial_m.
  \end{align*}
\end{proof}

\section{Wronskians and Linear Dependence} \label{sec:Wronskian} 

A {\bf generalized Wronskian} of a family $\vec{f} = (f_1,\ldots, f_n)$
of elements of $F$ is the determinant of a matrix of the form
\begin{equation}
  \begin{pmatrix}
    \partial_{m_1} f_1 &\cdots &\partial_{m_1} f_n \\ 
    \partial_{m_2} f_1 &\cdots &\partial_{m_2} f_n \\ 
    \vdots &\vdots &\vdots \\ 
    \partial_{m_n} f_1 &\cdots &\partial_{m_n} f_n
  \end{pmatrix}
  \label{eq:genW}
\end{equation}
where $\Omega(\vm) := (\Omega(m_1),\ldots, \Omega(m_n))$ is admissible. Here
a tuple of non-negative integers is {\bf admissible} if its $i$-th entry is
at most $i-1$ ($1 \le i \le n$). Consequently, $\Omega(m_1)$ must be 0 and so
$m_1$ must be $1$ if $\Omega(\vm)$ is admissible. The concept of generalized
Wronskian was introduced by Ostrowski~\cite{gw} and was used in the proof
of the famous Roth Lemma (see, for example,~\cite{DioGeo}*{Lemma~D.6.1,
Propsition~D.6.2}).

For notation simplicity, we will drop the fix but otherwise arbitrary
family ${\mathbf f}$ from the notation. Also, we will regard a matrix as
the tuple of its rows. Hence $(\partial_{\vm}) = (\partial_{m_1},\ldots,
\partial_{m_n})$ stands for the matrix in~\eqref{eq:genW}. Similarly, for a
tuple $\vk = (k_1,\ldots, k_n)$ of non-negative integers, $(\partial^{\vk})
= (\partial^{k_1},\ldots, \partial^{k_n})$ stands for the matrix
$(\partial^{k_i}f_j)$. The determinant of the matrix $(\partial^{i-1}f_j)$
($1 \le i,j \le n$) is the {\bf $\partial$-Wronskian} of $\vec{f}$, i.e. the
Wronskian of $\vec{f}$ with respect to the log-derivation.

The next proposition is a key step in our proof of Theorem~\ref{th:main}. The
idea is to ``replace'' each basic derivation in a generalized Wronskian
by the log-derivation successively while keeping the intermediate
determinants zero. To avoid this rather simple idea being obscured by the
induction argument, we encourage the reader to work out the case $n=3$
for himself/herself. In the course of writing out the proof, we find the
list-slicing syntax from the programming language Python useful\footnote{We
differ from Python slightly. For instance, in Python indices start at 0 and so
$\vk[:i]$ means $(k_0, \ldots, k_{i-1})$ instead.}: for $\vk = (k_1,\ldots,
k_n)$, let $\vk[: i]$ denote the tuple $(k_1,\ldots, k_i)$ and $\vk[i:]$
denote $(k_i,\ldots, k_n)$. We understood $\vk[:0]$ and $\vk[n+1:]$ to be
the null sequence.

\begin{prop}
  Let $K$ be a field containing a copy of $\Qq(\log(k) \colon k \in \Nn)$
  as subfield. If all generalized Wronskians of a family $f_1,\ldots, f_n
  \in F_K$ vanish then so does its $\partial$-Wronskian.  \label{p:all=>log}
\end{prop}
\begin{proof}
  For $1\le i \le n+1$, let $S_i(\vk)$ be the statement:
  \begin{align*}
    \det(\partial^{\vk[:i-1]},\partial_{\vm}) = 0
  \end{align*}
  for any $\vm$ with $\Omega(\vm) = \vk[i:]$. Let $S_i$ be the statement:
  for all admissible $\vk$, $S_i(\vk)$. The assumption on the vanishing
  of all generalized Wronskians means $S_1$ is true. If we can establish
  $S_{i+1}$ from $S_i$, then by induction $S_{n+1}$ is true. Consequently,
  $\det(\partial^0, \partial^1, \ldots, \partial^{n-1}) = 0$; that is the
  $\partial$-Wronskian of $\vec{f}$ vanishes since $\vk = (0,1,\ldots, n-1)$
  is admissible.

  To establish $S_{i+1}$, let $S_{i,j}(\vk)$ be the statement:
  \begin{align*}
    \det(\partial^{\vk[:i-1]}, \partial_{\ell}\partial^j, \partial_{\vm}) = 0
  \end{align*}
  for any $\ell \in \Nn$ with $\Omega(\ell) +j = k_i$ and for any $\vm$ with
  $\Omega(\vm) = \vk[i+1:]$. For $j \ge 0$, let $S_{i,j}$ assert $S_{i,j}(\vk)$
  for all admissible $\vk$ with $k_i \ge j$. Note that $S_{i,0}$ is equivalent
  to $S_i$ and hence is assumed to be true. Now assume $S_{i,j}$ is true for
  some $j \ge 0$. We argue that $S_{i,j+1}$ is true. Once this is established
  then it follows from induction that $S_{i,k_i}(\vk)$ is true for any
  admissible $\vk$. Since $S_{i,k_i}(\vk)$ is equivalent to $S_{i+1}(\vk)$
  and $\vk$ is an arbitrary admissible tuple, we establish $S_{i+1}$.
  
  It remains to prove $S_{i,j+1}$. To do that, pick an arbitrary admissible
  $\va$ with $a_i \ge j+1$, an arbitrary $b \in \Nn$ with $\Omega(b) + j +1 =
  a_i$ and an arbitrary $\vc$ with $\Omega(\vc) = \va[i+1:]$. We need to
  show that
  \begin{equation}
    \det(\partial^{\va[:i-1]}, \partial_{b}\partial^{j+1}, \partial_{\vc}) =0.
    \label{eq:s(i,j+1)}
  \end{equation}
  By Lemma~\ref{l:commutator},
  \begin{equation}
    \begin{split}
      \det&(\partial^{\va[:i-1]}, \partial_{b}\partial^{j+1},
    \partial_{\vc}) =\det(\cdots, \partial\partial_b\partial^j
    + \log(b)\partial_b\partial^j, \cdots)\\ &=\det(\cdots,
    \partial\partial_b\partial^j, \cdots) + \log(b)\det(\cdots,
    \partial_b\partial^j, \cdots).
  \end{split}
  \label{eq:sum}
  \end{equation}
  Note that $a_i -1 \ge j \ge 0$ and $\va':=(a_1,\ldots, a_i -1, \ldots a_n)$
  is still admissible. Because $\Omega(\va[:i-1]) = \Omega(\va'[:i-1])$,
  $\Omega(b) + j = a_i-1$ and $\Omega(\vc) = \va[i+1:] = \va'[i+1:]$ it
  follows from $S_{i,j}(\va')$ that
  \begin{equation}
    \det(\cdots, \partial_b\partial^j, \cdots) = 0
    \label{eq:2nd}
  \end{equation}
  We now claim that
  \begin{equation}
    \det(\cdots, \partial\partial_b\partial^j, \cdots) = 0
    \label{eq:1st}
  \end{equation}
  as well. This is because for each prime $p$, $\Omega(pb)+j = \Omega(b)+j+1 =
  a_i$ and so by $S_{i,j}(\va)$
  \begin{equation}
    \det(\partial^{\va[:i-1]}, \partial_p\partial_b\partial^j, \partial_{\vc})
    = \det(\partial^{\va[:i-1]}, \partial_{pb}\partial^j,
     \partial_{\vc})= 0.
    \label{eq:p}
  \end{equation}
  Now multiply Equation~\eqref{eq:p} by $\log(p)e_p$ then sum
  through the primes, we conclude from Lemma~\ref{l:log-basic} and the
  continuity of determinant that Equation~\eqref{eq:1st} holds. Finally,
  Equation~\eqref{eq:sum},~\eqref{eq:2nd} and~\eqref{eq:1st} together imply
  Equation~\eqref{eq:s(i,j+1)}. Since $\va$ is an arbitrary admissible tuple
  with $a_i \ge j+1$, we establish $S_{i,j+1}$ and complete the proof.
\end{proof}

We are now ready to prove Theorem~\ref{th:main}.
\begin{proof}[Proof of Theorem~\ref{th:main}]
  Suppose $\sum_i c_i f_i = 0$, where $c_1, \ldots, c_n \in K$ not all zero,
  witnessing the linear dependence of $f_1,\ldots, f_n$ over $K$. Since
  $\ker \partial_m \supseteq K$ for each $m > 1$ so
  \begin{align*}
   0 = \partial_m \sum c_if_i = c_1 \partial_m f_1 + \cdots + c_n\partial_m
   f_n.
  \end{align*}
  This equation holds for $m=1$ as well since $\partial_1$ is the identity
  operator. Thus, for any $\vm \in \Nn^n$, $\vec{c} = (c_1,\ldots, c_n)$ is
  a non-trivial solution to the linear system $(\partial_{\vm}\vec{f})\vec{x}
  = \vec{0}$. Hence every generalized Wronskian of $f_1,\ldots, f_n$ vanishes.
  
  To prove the other implication, let $L$ be a field that extends $K$
  and contains a copy of $\Qq(\log(k) \colon k \in \Nn)$. The existence
  of such $L$ is guaranteed by~\cite{bourbaki}*{Ch~5, Prop~4}. For any
  $f_1,\ldots, f_n \in F_K$, if their generalized Wronskians all vanish then,
  by Proposition~\ref{p:all=>log}, so does their $\partial$-Wronskian. It
  then follows from Theorem~\ref{th:wdf} that $f_1,\ldots, f_n$ are
  linearly dependent over the kernel of $\partial$ in $F_L$ which is $L$ by
  Proposition~\ref{p:constants}. But then they must also be linearly dependent
  over $K$. An elementary way of seeing this is as follows: view each $f_i$
  as a row vector $(f_{ij} \colon j \in \Nn)$ where $f_{ij}= f_i(j)$. If
  they were linearly independent over $K$, then by Gaussian elimination,
  there exists a sequence of elementary row operations over $K$ that turns
  $f_1,\ldots, f_n$ into a family of non-zero row vectors with straightly
  increasing orders (viewed as arithmetic functions). Since $L$ extends $K$,
  the same sequence of operations can be carried out over $L$ showing that
  $f_1,\ldots, f_n$ are linearly independent over $L$, a contradiction.
\end{proof}

If we view $f \in K\llrrbrac{X_1,\ldots, X_m}$ as an element of
$K\llrrbrac{X_i \colon i \in \Nn}$, it is clear that $(\partial/\partial
X_j) f = 0$ for all $j > m$. Hence, if the generalized Wronskians of a
family in $K\llrrbrac{X_1,\ldots, X_m}$ all vanish, then the same is true
when it is regarded as a family in $K\llrrbrac{X_i \colon i \in \Nn}$. Thus
Theorem~\ref{th:main} generalizes Theorem~3 of~\cite{bd}. At the same time,
Theorem~\ref{th:main} also generalizes Theorem~2.1 of~\cite{walker} since
the ring of germs of analytic functions at the origin of $K^m$ (where $K =
\Rr$ or $\Cc$) and hence its fraction field embeds into $F_K$. In the same
article, the author also identified the smallest collection of generalized
Wronskians whose vanishing implies linear dependence of the family over the
constants~\cite{walker}*{Theorem~3.1,~3.4}. This collection, in our set up,
are those generalized Wronskians indexed by tuples with the property that
a divisor of an entry is also an entry; for instance, in the case $n=3$,
the non-trivial assumptions are the vanishing of $\det(\id, \partial_p,
\partial_{p^2})$ and $\det(\id, \partial_p, \partial_q)$ for all primes $p$
and $q$. One can establish these results for the field $F_K$ by mimicking the
existing proofs. However, we are unable to do so using arguments similar to
those proposed here. Finally, we would like to mention that this article was
inspired by our study of another type dependence---algebraic dependence---of
arithmetic functions. The interplay between formal power series, arithmetic
functions and formal Dirichlet series is also proved to be fruitful in that
context~\cites{onalgind, imaf, pongax}.

\bibliographystyle{plain} 
\bibliography{wld}
\end{document}